\documentclass[12pt]{amsart}
\usepackage{amsmath}
\usepackage{amssymb}
\usepackage{amsthm}
\usepackage{array}
\usepackage{xy}
\usepackage[pdftex]{graphicx}
\usepackage{hyperref}
\usepackage{color}
\usepackage{transparent}
\usepackage{latexsym}
\usepackage{ulem}
\usepackage[margin=1.5in]{geometry}
\usepackage{amsmath,scalefnt}
\usepackage{amsthm}
\usepackage{amssymb}
\usepackage{graphicx}
\usepackage{setspace}
\usepackage{hyperref}

\numberwithin{equation}{section}

\newtheorem{thm}{Theorem}[section]

\newtheorem{lem}[thm]{Lemma}

\theoremstyle{plain}

\theoremstyle{plain}

\theoremstyle{definition}

\theoremstyle{remark}

\numberwithin{theorem}{section}
\numberwithin{equation}{section}
\numberwithin{figure}{section}

\begin{document}
\title[the geometry of the free boundary]{the geometry of the free boundary}
\author[Emanuel Indrei]{Emanuel Indrei \\
\\
\tiny {Department of Mathematics, Purdue University}\\
}

\makeatletter
\def\blfootnote{\xdef\@thefnmark{}\@footnotetext}
\makeatother

\date{}

\maketitle

\begin{abstract}
The non-transversal intersection of the free boundary with the fixed boundary is obtained for nonlinear uniformly elliptic operators
when $ \Omega = \{\nabla u \neq 0\} \cap \{x_n>0\}$ thereby solving a problem in elliptic theory that in the case of the Laplacian is completely understood but has remained arcane in the nonlinear setting in higher dimension. Also, a solution is given to a problem discussed in "Regularity of free boundaries in obstacle-type problems" \cite{MR2962060}. The free boundary is $C^1$ in a neighborhood of the fixed if the solution is physical and if $n=2$ in the absolute general context. The regularity is even new for the Laplacian.  The innovation is via geometric configurations on how free boundary points converge to the fixed boundary and investigating the spacing between free boundary points. 
\end{abstract}

\section{Introduction} 
In the classical obstacle problem, one seeks a function $u: B_1 \rightarrow \mathbb{R}$ which solves the PDE
\begin{equation}  \label{op}
\begin{cases}
\Delta u=\chi_{u > 0} & \text{in }B_1,\\
u \ge 0 & \text{in }B_1\\
u=g & \text{on }\partial B_1,\
\end{cases}
\end{equation}
where $g$ is a given function (e.g. continuous), 
and $B_1 \subset \mathbb{R}^n$ is the unit ball. The equation models the equilibrium state of an elastic membrane stretched over an ``obstacle" (i.e. a prescribed function).  In his seminal $1977$ paper, Caffarelli proved that the free boundary $\Gamma=\partial \{u>0\} \cap B_1$ is $C^1$ subject to a natural density assumption \cite{MR0454350}. In general, there are examples of two singular curves touching at a point (and which violate Caffarelli's assumption): Schaeffer constructed examples of cusp-like singularities \cite{MR0516201}.  

The behavior of the free boundary close to the fixed boundary for fully nonlinear elliptic operators has been open since Caffarelli's 1977 paper. By localizing around a point at the intersection of free and fixed boundary and considering a variation of \eqref{op} where the sign restriction on the solution is removed one is led to the following problem 
\begin{equation}  \label{meprev}
\begin{cases}
F(D^2 u)=\chi_\Omega & \text{in }B_{1}^{+}\\
u=0 & \text{on }B'_{1}
\end{cases}
\end{equation}
where $F$ is a convex, $C^1$ fully nonlinear uniformly elliptic operator, 
$$\Omega = \big(\{u \ne 0\} \cup \{\nabla u \neq 0\} \big) \cap \{x_n>0\}\subset \mathbb{R}_+^n.$$ 
In this context, the free boundary is $\Gamma=B_1^+\cap\partial \Omega$ and non-transversal intersection means that on a small enough scale, the free boundary is trapped below any cone with opening strictly less than $\pi$: for any $\epsilon>0$ there exists $\rho_\epsilon>0$ such that 
\[
\Gamma \cap B_{\rho_\epsilon}^+ \subset B_{\rho_\epsilon}^+ \setminus \mathcal{C}_\epsilon, 
\]
where $\mathcal{C}_\epsilon := \{x_n>\epsilon|x'|\}$, $x'=(x_1,\ldots, x_{n-1})$. 

The linear case has been well-studied via monotonicity formulas and explicit representations involving Green's function \cite{MR1745013, MR1950478, MR2962060}; see \cite[Chapter 8]{MR2962060} for many very interesting references. In the fully nonlinear uniformly elliptic context, non-transversal intersection remained a conjecture and was fully solved in \cite{I}. In addition, $C^1$ regularity of the free boundary was obtained in the one-phase problem without density assumptions which completed Caffarelli's regularity theory up-to-the boundary with a novel approach. 

One feature of interest is that the boundary case is generally not better than $C^{1,\text{Dini}}$, whereas in the interior one obtains analyticity at points with positive density. Also, some assumption is necessary in the interior (via Schaeffer's example) whereas close to the fixed boundary, the singularities cease existing.  

The problem for a choice of $\Omega$ in superconductivity is specifically stated in the book \cite[Chapter 8, Remark 8.8]{MR2962060} (open also for the Laplacian): investigating $C^1$ free boundary regularity at contact points for $\Omega = \{\nabla u \neq 0\}  \cap \{x_n>0\}$. Mathematically, this is complicated because one does not a priori know which values the function takes on in the complement of $\Omega$. 

The non-transversal intersection is proven in this paper (Theorem \ref{tt}). In addition, when $n=2$ the $C^1$ regularity is shown (Theorem \ref{1r}) and in the physical context without the dimension constraint (Theorem \ref{U1r}). The innovation is via geometric configurations on how free boundary points converge to the fixed boundary in case non-transversal intersection is false and investigating the spacing between free boundary points.

\section{Preliminaries} \label{pre}
The assumptions involve:
\begin{itemize}
\item $F(0)=0$;
\item $F$ is uniformly elliptic with
$$
\mathcal{P}^{-}(M-N)\le F(M)-F(N)\le\mathcal{P}^{+}(M-N),
$$
where $M$ and $N$ are symmetric matrices and $\mathcal{P}^{\pm}$
are the Pucci operators
$$
\mathcal{P}^{-}(M)=\inf_{\lambda_{0} \le N\le\lambda_{1}} \text{tr}(NM),\qquad\mathcal{P}^{+}(M)=\sup_{\lambda_{0}\le N\le\lambda_{1}}\text{tr} (NM);
$$
\item $F$ is convex and $C^1$ (the method also works for weaker assumptions).
\end{itemize}

Assume $\Omega'$ is an open set. A continuous function $u$ is an $L^p$- viscosity subsolution if for all $\phi \in W_{loc}^{2,p}(\Omega')$, whenever $\epsilon>0$, $A \subset \Omega'$ is open, $f \in L_{loc}^p(\Omega')$, $n<2p$, and 
$$
F(D^2\phi(x))-f(x) \ge \epsilon
$$
a.e. in $A$, $u-\phi$ cannot have a local maximum in $A$. A supersolution is similarly defined and a solution is both a supersolution and subsolution. An $L^n$-viscosity solution that lies in $W_{loc}^{2,p}(\Omega')$ satisfies the equation a.e. \cite{MR1376656}. Observe that one may extend the  theory to the parabolic case
\cite{MR4322624}.\\

A continuous function $u$ belongs to $P_r^+(0,M, \Omega)$ if  \\

\noindent 1. $F(D^2 u)=\chi_\Omega$ a.e. in $B_r^+$\\
2. $||u||_{L^\infty(B_r^+)} \le M$\\
3. $u=0$ on $\{x_n=0\} \cap \overline{B_r^+}=:B'_{r}$\\

in the viscosity sense.\\

Also, given $u \in P_r^+(0,M, \Omega)$, the free boundary is denoted by 
$
\Gamma=\partial \Omega \cap B_r^+;
$ 
a blow-up limit of $\{u_j\} \subset P_1^+(0,M,\Omega)$ is a limit of the form 

\begin{equation*} 
\lim_{k \rightarrow \infty} \frac{u_{j_k}(s_kx)}{s_k^2},
\end{equation*}
where $\{j_k\}$ is a subsequence of $\{j\}$ and $s_k \rightarrow 0^+$.

\section{Non-transversal intersection and the $C^1$ regularity}
\begin{thm} \label{tt}
Suppose $u \in P_1^+(0,M, \Omega)$. There exists $r_0>0$ and a modulus of continuity $\omega$ such that 
$$\Gamma(u) \cap B_{r_0}^+ \subset \{x=(x',x_n): x_n \le \omega(|x'|)|x'|\}$$ provided $0 \in \overline{\Gamma(u)}$.  
\end{thm}

\begin{proof}
Suppose there exists $\epsilon>0$ such that for all $l \in \mathbb{N}$ there exists $x_l \in C_\epsilon \cap B_{\frac{1}{l}}^+\cap \Gamma(u)$. Consider $e \in \mathbb{S}^{n-1}$ such that along a subsequence $x_l \rightarrow 0$ along $e$ ($e_n>0$); observe 
$$
\frac{x_l}{|x_l|} \rightarrow e \in \overline{C_\epsilon}.
$$
If $u_0=ax_1x_n+bx_n^2$ is a blow-up limit of $\{u\}$ for $a \neq 0$ mod a rotation, $a>0$, 
$$
\lim_{k\rightarrow \infty} \frac{u(s_kx)}{s_k^2}=ax_1x_n+bx_n^2.
$$
Next, observe that when
$$
\tilde u_k(x)= \frac{u(s_kx)}{s_k^2},
$$
one obtains 
$$
\sup_k ||D^2 \tilde u_k||_{L^\infty(B_{t_*}^+)}=T=T(t_*)<\infty
$$
if $t_*>0$ \cite{MR3513142}.
Let $z_{l,k}=\frac{x_l}{s_k}$, $R>1$, and $\alpha \in (0,1/5).$ 
Since  
$$
||\tilde u_k - u_0||_{C^{1,\alpha_r}(B_R^+)} \rightarrow 0, 
$$ and 
$$
\inf_{C_\epsilon \setminus B_{\alpha}^+} |\nabla u_0 (x)| \ge a_1>0
$$
it follows that 
there exists $N=N(\alpha, R, a, T)$ where assuming $k \ge N$,
\begin{equation} \label{w}
\inf_{(B_R^+\cap C_\epsilon) \setminus B_{\alpha}^+} |\nabla \tilde u_k (x)|  \ge \frac{a_1}{3}.
\end{equation}
Let $k_1\ge N$ be the smallest integer for which $|z_{1,k_1}| \ge R$  via $s_k \rightarrow 0^+$; note that \eqref{w} readily implies 
$$
\Gamma_k \cap ((B_R^+\cap C_\epsilon) \setminus B_{\alpha}^+) =\emptyset
$$
when $k \ge N$.
Now, since 
$$\mathcal{L}_t=\{t z_{1,k_1}: t \in (0,1] \} \subset C_\epsilon$$
set $z_{k_1(1)}$ as an element on $B_{|x_1|}^+ \cap \Gamma(\tilde u_{k_1})$ closest to $\mathcal{L}_t$; note that $|x_l| \le \frac{1}{l} \rightarrow 0$, therefore without loss of generality, one may assume $|x_1|<\alpha$. In particular, if 
$$
A_{1}=\{tz_{1,k_1}+(1-t)z_{k_1(1)}: 0\le t\le 1\}
$$
then there are at least two free boundary elements on $A_1$: $z_{1,k_1}$, $z_{k_1(1)}$. Note that mod a re-labeling, thanks to
$$
\Gamma_k \cap ((B_R^+\cap C_\epsilon) \setminus B_{\alpha}^+) =\emptyset,
$$
one may assume $z_{1,k_1}, z_{k_1(1)}$ are the sole free boundary points on $A_1$ because if not, since there are no free boundary points in $((B_R^+\cap C_\epsilon) \setminus B_{\alpha}^+)$, one can pick $z_1 \in A_1$ closest to $\partial B_R$, and then one can similarly choose $z_2 \in A_{1}$ closest to $\partial B_\alpha$. In particular, 
$$t \mapsto \partial_{x_1} \tilde u_{k_1}(tz_{l_1,k_1}+(1-t)z_{k_1(1)})$$ 
is absolutely continuous on $[0,1]$. 
Next denote $l_1=1$ and select the smallest integer $l_2 \ge 2$ for which $|x_{l_2}| <|x_1|$. 
Let $k_2 \ge k_1$ be the smallest integer for which $|z_{l_2,k_2}| > R$: $\frac{|x_{l_2}|}{s} \rightarrow \infty$ when $s \rightarrow 0^+$, therefore via $s_k \rightarrow 0$, $k_2$ exists. Hence one may iterate and thus obtain two elements on $\Gamma(\tilde u_{k_2})$:  

$$z_{l_2,k_2} \in C_\epsilon \setminus B_{R}^+, $$ 

$$z_{k_2(2)} \in B_{|x_{l_2}|}^+.$$
Observe that via iterating one can select $z_{l_m,k_m}, z_{k_m(m)} \in \Gamma(\tilde u_{k_m})$ s.t.
$$e^m=\frac{z_{l_m,k_m}-z_{k_m(m)}}{|z_{l_m,k_m}-z_{k_m(m)}|},$$

$$
|z_{l_m,k_m}| \ge R,
$$
$$
|z_{k_m(m)}| \le \alpha,
$$
$$
|z_{l_m,k_m}-z_{k_m(m)}| \ge R-\alpha>0;
$$
moreover, via
$$
\frac{x_l}{|x_l|} \rightarrow e,
$$
note
$$
\frac{z_{l_m,k_m}}{|z_{l_m,k_m}|}=\frac{x_{l_m}}{|x_{l_m}|} \rightarrow e,
$$
$$
\frac{z_{k_m(m)}}{|z_{k_m(m)}|} \rightarrow e
$$
(because $z_{l,k}=\frac{x_l}{s_k}$, one may select free boundary points $z_{k_m(m)}$ close to the line along $e$),
thus
\begin{align*}
&\Big|\frac{z_{l_m,k_m}-z_{k_m(m)}}{|z_{l_m,k_m}-z_{k_m(m)}|}-e\Big|\\
&=\Big|\frac{z_{l_m,k_m}-|z_{l_m,k_m}|e+|z_{l_m,k_m}|e-|z_{k_m(m)}|e+|z_{k_m(m)}|e -z_{k_m(m)}}{|z_{l_m,k_m}-z_{k_m(m)}|}-e\Big|\\
&=\Big|\frac{|z_{l_m,k_m}|(\frac{z_{l_m,k_m}}{|z_{l_m,k_m}|}-e)+|z_{l_m,k_m}|e-|z_{k_m(m)}|e+|z_{k_m(m)}|(e -\frac{z_{k_m(m)}}{|z_{k_m(m)}|})}{|z_{l_m,k_m}-z_{k_m(m)}|}-e\Big| \\
&\le \Big|\frac{|z_{l_m,k_m}|(\frac{z_{l_m,k_m}}{|z_{l_m,k_m}|}-e)+|z_{k_m(m)}|(e -\frac{z_{k_m(m)}}{|z_{k_m(m)}|})}{|z_{l_m,k_m}-z_{k_m(m)}|}\Big|+\Big|\frac{|z_{l_m,k_m}|-|z_{k_m(m)}|-|z_{l_m,k_m}-z_{k_m(m)}|}{|z_{l_m,k_m}-z_{k_m(m)}|}\Big|\\
& \rightarrow 0 
\end{align*}
since assuming $|z_{l_m,k_m}| \rightarrow \infty$, via $|z_{k_m(m)}| \le \alpha$,
$$
\Big|\frac{1-\frac{|z_{k_m(m)}|}{|z_{l_m,k_m}|}-|\frac{z_{l_m,k_m}}{|z_{l_m,k_m}|}-\frac{z_{k_m(m)}}{|z_{l_m,k_m}|}|}{|\frac{z_{l_m,k_m}}{|z_{l_m,k_m}|}-\frac{z_{k_m(m)}}{|z_{l_m,k_m}|}|}\Big|  \rightarrow 0,
$$
$$
\frac{|z_{l_m,k_m}|}{|z_{l_m,k_m}-z_{k_m(m)}|} \rightarrow 1.
$$
And, assuming (mod a subsequence) $|z_{l_m,k_m}| \rightarrow \alpha_1<\infty$, since
$$
|z_{l_m,k_m}-z_{k_m(m)}| \ge R-\alpha,
$$
it then 
immediately follows that
$$
\Big|\frac{|z_{l_m,k_m}|(\frac{z_{l_m,k_m}}{|z_{l_m,k_m}|}-e)+|z_{k_m(m)}|(e -\frac{z_{k_m(m)}}{|z_{k_m(m)}|})}{|z_{l_m,k_m}-z_{k_m(m)}|}\Big|
\rightarrow 0,
$$
also via
$$
\frac{z_{l_m,k_m}}{|z_{l_m,k_m}|} \rightarrow e,
$$
$$
\frac{z_{k_m(m)}}{|z_{k_m(m)}|} \rightarrow e,
$$
note
$$
\Big||z_{l_m,k_m}|-|z_{k_m(m)}|-|z_{l_m,k_m}-z_{k_m(m)}|\Big| \rightarrow 0.
$$
Therefore
$$
e^m \rightarrow e.
$$ 
Let
$$
A_{k_m, l_m}=\{tz_{l_m,k_m}+(1-t)z_{k_m(m)}: 0\le t\le 1\},
$$

$$
A_{k_m, l_m, \alpha}=\{tz_{l_m,k_m}+(1-t)z_{k_m(m)}: 0\le t\le t(\alpha), \hskip .1in t(\alpha) =\sup \{t: tz_{l_m,k_m}+(1-t)z_{k_m(m)} \in B_{\alpha}^+\}\},
$$

$$
A_{k_m, l_m, R, \alpha}=\{tz_{l_m,k_m}+(1-t)z_{k_m(m)}: t(\alpha)\le t\le t(R), \hskip .1in t(R) =\sup \{t: tz_{l_m,k_m}+(1-t)z_{k_m(m)} \in B_{R}^+\}\},
$$

$$
A_{k_m, l_m, R}=\{tz_{l_m,k_m}+(1-t)z_{k_m(m)}: t(R)\le t\le 1\}.
$$
Observe
\begin{align*}
\frac{d}{dt} \partial_{x_1}& \tilde u_{k_m}(tz_{l_m,k_m}+(1-t)z_{k_m(m)})\\
& = \langle \nabla \partial_{x_1} \tilde u_{k_m}(tz_{l_m,k_m}+(1-t)z_{k_m(m)}), z_{l_m,k_m}-z_{k_m(m)} \rangle \\
&=  \langle \nabla \partial_{x_1} \tilde u_{k_m}(tz_{l_m,k_m}+(1-t)z_{k_m(m)}), \frac{z_{l_m,k_m}-z_{k_m(m)}}{|z_{l_m,k_m}-z_{k_m(m)}|} \rangle |z_{l_m,k_m}-z_{k_m(m)}|\\
&= \langle \nabla \partial_{x_1} \tilde u_{k_m}(tz_{l_m,k_m}+(1-t)z_{k_m(m)}), e^m  \rangle |z_{l_m,k_m}-z_{k_m(m)}|\\
& = \partial_{e^m x_1} \tilde u_{k_m}(tz_{l_m,k_m}+(1-t)z_{k_m(m)}) |z_{l_m,k_m}-z_{k_m(m)}|.
\end{align*}
In particular, thanks to absolute continuity of
$$t \mapsto \partial_{x_1} \tilde u_{k_m}(tz_{l_m,k_m}+(1-t)z_{k_m(m)})$$ 
on $[0,1]$ and the fact that $z_{l_m,k_m}, z_{k_m(m)} \in \Gamma(\tilde u_{k_m})$, 
\begin{align*}
0&=\partial_{x_1} \tilde u_{k_m}(tz_{l_m,k_m}+(1-t)z_{k_m(m)})|_{t=1}-\partial_{x_1} \tilde u_{k_m}(tz_{l_m,k_m}+(1-t)z_{k_m(m)})|_{t=0}\\
&=\int_0^1 \partial_{e^m x_1} \tilde u_{k_m}(sz_{l_m,k_m}+(1-s)z_{k_m(m)}) |z_{l_m,k_m}-z_{k_m(m)}| ds.
\end{align*}
Therefore, since $|z_{l_m,k_m}-z_{k_m(m)}| \ge R-\alpha>0,$
\begin{align*}
0&=\int_0^1 \partial_{e^m x_1} \tilde u_{k_m}(sz_{l_m,k_m}+(1-s)z_{k_m(m)}) ds\\
&=\int_0^{t(\alpha)} \partial_{e^m x_1} \tilde u_{k_m}(sz_{l_m,k_m}+(1-s)z_{k_m(m)}) ds\\
&+\int_{t(\alpha)}^{t(R)} \partial_{e^m x_1} \tilde u_{k_m}(sz_{l_m,k_m}+(1-s)z_{k_m(m)})  ds\\
&+\int_{t(R)}^{1} \partial_{e^m x_1} \tilde u_{k_m}(sz_{l_m,k_m}+(1-s)z_{k_m(m)})  ds.
\end{align*}
Supposing $s \in [t(\alpha), t(R)]$, the $C^{2,\alpha}(B_R^+\setminus B_\alpha^+)$ regularity yields 

$$
\partial_{e^m x_1} \tilde u_{k_m}(sz_{l_m,k_m}+(1-s)z_{k_m(m)}) \rightarrow \partial_{x_1e} u_{0}=e_na>0
$$
uniformly; therefore
$$
\int_{t(\alpha)}^{t(R)} \partial_{e^m x_1} \tilde u_{k_m}(sz_{l_m,k_m}+(1-s)z_{k_m(m)})  ds \rightarrow (t(R)-t(\alpha))e_na>0.
$$
Supposing $s \in [t(R), 1]$, then Fatou's lemma implies

$$
0\le \liminf_m (1-t(R)) \partial_{x_1e} u_{0} \le \liminf_m   \int_{t(R)}^{1} \partial_{e^m x_1} \tilde u_{k_m}(sz_{l_m,k_m}+(1-s)z_{k_m(m)})ds;$$

also, $C^{1,1}(B_1^+)$ regularity yields
$$
\int_0^{t(\alpha)} |\partial_{e^m x_1} \tilde u_{k_m}(sz_{l_m,k_m}+(1-s)z_{k_m(m)}) |ds \le T t(\alpha);
$$
observe $t(\alpha)\rightarrow 0^+$ as $\alpha \rightarrow 0$; selecting $R>1$ and some small $\alpha(R,T, e,a) \in (0,1/5)$ so that
$$
\frac{t(R) e_n a}{e_na+T}>t(\alpha),
$$

\begin{align*}
0&=\liminf_{m} \Big(\int_0^1 \partial_{e^m x_1} \tilde u_{k_m}(sz_{l_m,k_m}+(1-s)z_{k_m(m)}) ds\Big)\\
& \ge  (t(R)-t(\alpha))e_na-T t(\alpha),
\end{align*}
a contradiction. 
This then yields that 
$$u_0=ax_1x_n+bx_n^2$$
with $a \neq 0$, cannot be generated as a blow-up limit of $\{u\}$. Therefore, \cite[Proposition 3.6]{I} implies that all blow-up limits are of the form 
$$u_0=bx_n^2.$$
To complete the argument, let $x_l \in \Gamma(u) \cap B_{1/l}^+ \cap \mathcal{C}_\epsilon$, $y_l=\frac{x_l}{r_k}$ with $r_l=|x_l|$. Observe when $\tilde u_l(x)=\frac{u(r_lx)}{r_l^2}$ that $y_l \in \Gamma(\tilde u_l)$, $\tilde u_l \rightarrow bx_n^2$, $y_{l} \rightarrow y \in \partial B_1 \cap C_\epsilon$ (up to a subsequence), and $y \in \Gamma(u_0)$, a contradiction. 

\end{proof}

The following lemma is the central tool to show $C^1$ regularity.

\begin{lem} \label{rl}
Suppose $u \in P_1^+(0,M, \Omega)$ and  $ u\ge 0.$
Then for any $\epsilon>0$ there exists $r(\epsilon, M)>0$ such that if $x^0 \in \Gamma(u) \cap B_{1/2}^+$ and $d=x_n^0<r,$ then 
$$ 
\sup_{B_{2d}^+(x^0)} |u-h| \le \epsilon d^2, \hskip .2in \sup_{B_{2d}^+(x^0)} | \nabla u - \nabla h| \le \epsilon d, 
$$
where $$h(x)= b[(x_n-d)^+]^2,$$ and $b>0$ depends on the ellipticity constants of $F$. 
\end{lem}

\begin{proof}
Suppose the statement is not true, then there exists $\epsilon>0$, non-negative $u_j \in P_1^+(0,M, \Omega)$, and $x^j \in \Gamma(u_j) \cap B_{1/2}^+$, $d_j=x_n^j \rightarrow 0$, so that
$$\sup_{B_{2d_j}(x^j)^+} |u_j- b[(x_n-d_j)^+]^2|>\epsilon d_j^2,$$ or 
$$\sup_{B_{2d_j}(x^j)^+} |\nabla u_j- 2b(x_n-d_j)^+|>\epsilon d_j.$$   
Denote $$\tilde u_j(x)=\frac{u_j((x^j)'+d_jx)}{d_j^2}=\frac{u_j((x^j)'+d_jx)-u_j((x^j)')}{d_j^2},$$
$u_j((x^j)')=0$ because of the Dirichlet boundary condition, so that
$$||\tilde u_j - h||_{C^1(B_2^+(e_n))} \ge \epsilon,$$ $h(x)=b[(x_n-1)^+]^2$. Since $|\nabla \tilde u_j(e_n)|=0,$ the $C^{1,1}$ regularity of $\tilde u_j$ implies that assuming $|x| \le w$

\begin{align*}
|\tilde u_j(x)|&=|\frac{u_j((x^j)'+d_jx)-u_j((x^j)')}{d_j^2}|\\
&\le \frac{|\nabla u_j(r_{x^j})||x|}{d_j} =\frac{|\nabla u_j(r_{x^j})-\nabla u_j(x^j)||x|}{d_j} \le T_1,
\end{align*}
thanks to $\nabla u_j(x^j)=0$; therefore $|\tilde u_j(x)|$ is locally bounded. Via Arzela-Ascoli and the $C^{1,1}$ regularity,
by passing to a subsequence, if necessary, $$\tilde u_j \rightarrow u_0$$ where $u_0 \in C^{1,1}(\mathbb{R}_+^n)$ satisfies the following PDE in the viscosity sense
\begin{equation} \label{m3}
\begin{cases}
F(D^{2}u_0)=1 & \text{a.e. in }\mathbb{R}_+^n\cap\Omega_0,\\
|\nabla u_0|=0 & \text{in }\mathbb{R}_+^n\backslash\Omega_0,\\
u_0=0 & \text{on }\mathbb{R}_+^{n-1}.
\end{cases}
\end{equation}   
Now suppose $R>4$,
$$N=N(R)= \limsup_{y_n\rightarrow 0, y_n>0, y_n \in \{y_n: (y',y_n) \in \cup_j \Gamma(u_j)\}}\sup_{\{x: x_n>0, |x| \le R\}}  \sup_{e \in \mathbb{S}^{n-2} \cap e_n^{\perp}}  \sup_{y \in \overline{B_{1/2}^+} \cap \{y':(y',y_n) \in \Gamma(u_j)\}} \frac{ \partial_e u_j(y_nx+y)}{x_ny_n}$$
and note that $\sup_{R>4} N(R) \le N_*<\infty$ by $C^{1,1}$ regularity and the boundary condition: for any 
$e \in \mathbb{S}^{n-2} \cap e_n^{\perp}$ and $y \in \overline{B_{1/2}^+} \cap \{x_n=0\}$, it follows that $\partial_{e} u_j((y_nx)'+y)=0$. In addition, 

\begin{equation} \label{part}
N \ge \lim_k \bigg| \frac{\partial_{x_i}  u_k(d_k x+(x^k)')}{d_k x_n}\bigg| = \lim_k \bigg| \frac{\partial_{x_i} \tilde u_k(x)}{x_n} \bigg| =\bigg|\frac{\partial_{x_i} u_0(x)}{x_n}\bigg|
\end{equation}
for all $i \in \{1,\ldots, n-1\}$. In particular, let $v= \partial_{x_1} u_0$ so that 
\begin{equation} \label{ine}
|v(x)| \le Nx_n
\end{equation}
 in $B_R^+$.
If $N=0$, then $\partial_{x_i} u_0=0$ for all $i \in \{1,\ldots,n-1\}$ and  $u_0(x)=u_0(x_n)$. Since $e_n$ is a free boundary point, one obtains $u_0=h$, a contradiction via
$$||\tilde u_k - h||_{C^1(B_2^+(e_n))} \ge \epsilon.$$
Therefore $N>0$  
and there is a sequence 
$\{x^k\}_{k \in \mathbb{N}}$ with $x_n^k>0$,  $y^k \in \overline{B_{1/2}^+} \cap \{x_n=0\}\cap \{y':(y',y_n) \in \Gamma(u_k)\}$, and $e^k \in \mathbb{S}^{n-2} \cap e_n^{\perp}$ such that $x^k=(y^k)_nw_k$, $w_k \in B_R^+$,
$$N=\lim_{k \rightarrow \infty} \frac{1}{x_n^k}  \partial_{e^{k}} u_k(x^k+y^k).$$ 
By compactness, along a subsequence $e^k \rightarrow e_1 \in \mathbb{S}^{n-2}$, $w_k \rightarrow w_a  \in \overline{B_R^+}$  so that up to a rotation
$$N= \lim_{k \rightarrow \infty} \frac{1}{x_n^k}  \partial_{x_1} u_k(x^k+y^k).$$ 
The free boundary points $(y^k,(y^k)_n)$ imply $\partial_{x_1} u_k((y^k,(y^k)_n))=0$, therefore because $(y^k)_n\rightarrow 0$, $|w_a|>0$ because for any $(w_k)_n \rightarrow 0$, select $\overline{w}_k=(\overline{w}_k', (w_k)_n)$, where $\overline{w}_k' \nrightarrow 0$; set $\overline{x}^k=(y^k)_n\overline{w}_k,$ then (assuming $k$ is large), $(\overline{x}^k)_n= x_n^k$, but one can choose $\overline{w}_k'$ so that
$$0=\partial_{x_1} u_k((y^k,(y^k)_n)) \approx \partial_{x_1} u_k(x^k+y^k) \le \partial_{x_1} u_k(\overline{x}^k+y^k),$$
in particular $N$ will prefer $\overline{x}^k$ instead of $x^k$ (via the maximality).
Let 
$$\tilde u_k(x)=\frac{u_k(y^k+r_kx)}{r_k^2},$$ $r_k=|x^k|$, $z^k=\frac{x^k}{r_k}$, and note that along a subsequence $z^k \rightarrow z \in \mathbb{S}^{n-1}$ and $\tilde u_k \rightarrow u_0$. It follows that $\partial_{x_1}u_0(z)=Nz_n$ and thanks to $N(R) \le N_*$, one may let $R \rightarrow \infty$. In particular, proceeding as in  \cite{MR3513142}, $u_0(x)=ax_1x_n+cx_n+\tilde b x_n^2$ for $a \neq 0$ and $c, \tilde b \in \mathbb{R}$, contradicting that  $u_k $ is non-negative.   
\end{proof}

\begin{thm}  \label{U1r}
Let $u \in P_1^+(0, M, \Omega)$,  $0 \in \overline{\Gamma(u)},$ and assume that $u \ge 0$.
Then there exists $r_0>0$ such that $\Gamma$ is the graph of a $C^1$ function in $B_{r_0}^+$. 
\end{thm}

\begin{proof}
The lemma implies that the free boundary is Lipschitz near the origin. Also, if $\epsilon>0$, there is $\alpha_\epsilon>0$ such that for $y \in B_{\alpha_\epsilon}^+$ with $N_y$ a normal, 
$$
|N_y - e_n|<\epsilon.
$$
The non-transversal intersection yields that the normal exists at the origin and $N_0=e_n.$ Observe that there exists $r_0>0$ such that for $x \in B_{r_0}^+$, $\epsilon>0$, there is $a_\epsilon>0$ such that if $|x-z|<a_\epsilon$, 
$$
|N_x-N_z|<\epsilon.
$$
Assume not, then there exists $\epsilon>0$, $x_k, z_k$, $|z_k| \rightarrow 0$, $|x_k| \rightarrow 0$, 
$$|N_{x_k}-N_{z_k}| \ge \epsilon.$$
Observe that via $\alpha_{\frac{\epsilon}{2}}>0$, 
$z_k,x_k \in B_{\alpha_\epsilon}^+$
supposing $k$ large; in particular,

$$
|N_{z_k} - e_n|<\frac{\epsilon}{2},
$$

$$
|N_{x_k} - e_n|<\frac{\epsilon}{2},
$$
thus that contradicts
$$
\epsilon \le |N_{z_k} - N_{x_k}|<\epsilon.
$$

\end{proof}

In the physical case, the blow-up 
$$
u_0(x)=ax_1x_n+cx_n+\tilde b x_n^2,
$$
$a \neq 0$, is ruled out thanks to $u \ge 0$. Supposing $n=2$, one may rule out the undesired blow-up in general with a maximum principle \cite{MR3986537}.  

\begin{lem} \label{rln}
Suppose $u \in P_1^+(0,M, \Omega)$ and  $n=2.$
Then for any $\epsilon>0$ there exists $r(\epsilon, u)>0$ such that if $x^0 \in \Gamma(u) \cap B_{1/2}^+$ and $d=x_2^0<r,$ then 
$$ 
\sup_{B_{2d}^+(x^0)} |u-h| \le \epsilon d^2, \hskip .2in \sup_{B_{2d}^+(x^0)} | \nabla u - \nabla h| \le \epsilon d, 
$$
where $$h(x)= b[(x_2-d)^+]^2,$$ and $b>0$ depends on the ellipticity constants of $F$. 
\end{lem}

\begin{proof}
Suppose not, then there exists $\epsilon>0$, $x^k \in \Gamma(u) \cap B_{1/2}^+$ with $d_k=x_n^k \rightarrow 0$, for which 
$$\sup_{B_{2d_k}(x^k)^+} |u- b[(x_n-d_k)^+]^2|>\epsilon d_k^2,$$ or 
$$\sup_{B_{2d_k}(x^k)^+} |\nabla u- 2b(x_n-d_k)^+|>\epsilon d_k.$$   
Let $$\tilde u_k(x)=\frac{u((x^k)'+d_kx)}{d_k^2}=\frac{u((x^k)'+d_kx)-u((x^k)')}{d_k^2},$$
$u((x^k)')=0$ because of the Dirichlet boundary condition, so that in particular 
$$||\tilde u_k - h||_{C^1(B_2^+(e_n))} \ge \epsilon,$$ where $h(x)=b[(x_n-1)^+]^2$. Since $|\nabla \tilde u_k(e_n)|=0,$ the $C^{1,1}$ regularity of $\tilde u_k$ implies that assuming $|x| \le w$,

\begin{align*}
|\tilde u_k(x)|&=|\frac{u((x^k)'+d_kx)-u((x^k)')}{d_k^2}|\\
&\le \frac{|\nabla u(r_{x^k})||x|}{d_k} =\frac{|\nabla u(r_{x^k})-\nabla u(x^k)||x|}{d_k} \le T_1,
\end{align*}
thanks to $\nabla u(x^k)=0$; therefore $|\tilde u_k(x)|$ is locally bounded. Via Arzela-Ascoli,
by passing to a subsequence, if necessary, $$\tilde u_k \rightarrow u_0$$ where $u_0 \in C^{1,1}(\mathbb{R}_+^n)$ satisfies the following PDE in the viscosity sense
\begin{equation} \label{m3}
\begin{cases}
F(D^{2}u_0)=1 & \text{a.e. in }\mathbb{R}_+^n\cap\Omega_0,\\
|\nabla u_0|=0 & \text{in }\mathbb{R}_+^n\backslash\Omega_0,\\
u_0=0 & \text{on }\mathbb{R}_+^{n-1}.
\end{cases}
\end{equation}   
Now suppose $R>4$, 
$$N=N(R)=\limsup_{y_n\rightarrow 0, y_n>0, y_n \in \{y_n: (y',y_n) \in \Gamma\}}\sup_{\{x: x_n>0, |x| \le R\}}   \sup_{e \in \mathbb{S}^{n-2} \cap e_n^{\perp}}  \sup_{y \in \overline{B_{1/2}^+}  \cap \{y':(y',y_n) \in \Gamma\}} \frac{\partial_e u(y_nx+y)}{x_ny_n}$$
and note that $\sup_{R>4} N(R) \le N_*<\infty$ by $C^{1,1}$ regularity and the boundary condition: for any 
$e \in \mathbb{S}^{n-2} \cap e_n^{\perp}$ and $y \in \overline{B_{1/2}^+} \cap \{x_n=0\}$, it follows that $\partial_{e} u((y_nx)'+y)=0$. In addition, 

\begin{equation} \label{part}
N \ge \lim_k \bigg| \frac{\partial_{x_i}  u(d_k x+(x^k)')}{d_k x_n}\bigg| = \lim_k \bigg| \frac{\partial_{x_i} \tilde u_k(x)}{x_n} \bigg| =\bigg|\frac{\partial_{x_i} u_0(x)}{x_n}\bigg|
\end{equation}
for all $i \in \{1,\ldots, n-1\}$. In particular, let $v= \partial_{x_1} u_0$ so that 
\begin{equation} \label{ine}
|v(x)| \le Nx_n
\end{equation}
 in $B_R^+$.
If $N=0$, then $\partial_{x_i} u_0=0$ for all $i \in \{1,\ldots,n-1\}$ and  $u_0(x)=u_0(x_n)$. Since $e_n$ is a free boundary point, one obtains $u_0=h$, a contradiction via
$$||\tilde u_k - h||_{C^1(B_2^+(e_n))} \ge \epsilon.$$
Therefore $N>0$  
and there is a sequence as in  Lemma \ref{rl},
$\{x^k\}_{k \in \mathbb{N}}$ with $x_n^k>0$,  $y^k \in \overline{B_{1/2}^+} \cap \{x_n=0\}\cap \{y':(y',y_n) \in \Gamma\}$, and $e^k \in \mathbb{S}^{n-2} \cap e_n^{\perp}$ such that $x^k=(y^k)_nw_k$, $w_k \in B_R^+$,
$$N=\lim_{k \rightarrow \infty} \frac{1}{x_n^k}  \partial_{e^{k}} u(x^k+y^k).$$ 
By compactness, along a subsequence $e^k \rightarrow e_1 \in \mathbb{S}^{n-2}$, $w_k \rightarrow w_a$ ($|w_a|>0$) so that up to a rotation
$$N= \lim_{k \rightarrow \infty} \frac{1}{x_n^k}  \partial_{x_1} u(x^k+y^k).$$ Let 
$$\tilde u_k(x)=\frac{u(y^k+r_kx)}{r_k^2},$$ $r_k=|x^k|$, $z^k=\frac{x^k}{r_k}$, and note that along a subsequence $z^k \rightarrow z \in \mathbb{S}^{n-1}$ and $\tilde u_k \rightarrow u_0$. It follows that $\partial_{x_1}u_0(z)=Nz_n$ and thanks to $N(R) \le N_*$, one may let $R \rightarrow \infty$. In particular, proceeding as in  Lemma \ref{rl}, it follows that $u_0(x)=ax_1x_n+cx_n+\tilde b x_n^2$ for $a > 0$ and $c, \tilde b \in \mathbb{R}$.   
Note that 
$$
\nabla u_0(x)= (ax_n, 0\ldots 0, ax_1+c+2\tilde b x_n).
$$
In particular, supposing $n=2$,
$$
\nabla u_0(x)=0
$$
iff $x_{a,c}=(\frac{-c}{a},0)$. Set $R \ge 1$ and note that since $\tilde u_k \rightarrow u_0$ in $C_{loc}^{1,\alpha}$, there exists $k_R \in \mathbb{N}$ such that for $k \ge k_R$, $|\nabla \tilde u_k|>\tilde c$ in 
$$E=\big(\{-R+\frac{-c}{a}<x_1<-\frac{R}{2}+\frac{-c}{a}\}\cup \{\frac{R}{2}+\frac{-c}{a}<x_1<R+\frac{-c}{a}\}\big)\cap B_{2R}^+(x_{a,c}),$$ and therefore, $\tilde u_k \in C^{2,\alpha}(E)$. Up to a subsequence, 
$$\frac{|\partial_{x_1} \tilde u_k-ax_2|}{x_2}=m_k \rightarrow 0$$ so that
$$
\partial_{x_1} \tilde u_k \ge \frac{a}{2}x_2
$$ 
in $E$ for $k \ge \tilde k_R \in \mathbb{N}$. Next pick $\eta \in (0,R)$ and $k_0'$ so that if $k \ge k_0'$, 
$$\partial_{x_1} \tilde u_k(x) \ge \frac{a \eta}{2}$$ for $x \in \{x_n \ge \eta \}$. 
Observe $y_j=(y^j,(y^j)_n) \in \Gamma$, hence 
$$
\frac{y_j-y^k}{r_k}
$$
is a free boundary element for $\tilde u_k$,  $k \ge k_0'$, and $j \in \mathbb{N}$. Now, $\nabla u_0(x)=0$ iff $x_{a,c}=(\frac{-c}{a},0)$ then implies that assuming $j$ large ($k$ large and fixed), via $(y^j)_n \rightarrow 0$,
$$
\frac{y_j-y^k}{r_k} \in B_{R/4}^+(x_{a,c}).
$$
Therefore 
$$
\frac{y_j-y^k}{r_k} \in \Gamma(\tilde u_k)
$$
then yields 
$$\text{Int}\{ \nabla \tilde u_k =0 \} \cap B_{R/2}^+(x_{a,c})\neq \emptyset$$ 
since supposing this is  not true, non-degeneracy yields  $\tilde u_{k} \in C^{2,\alpha}(B_{R/4}^+(x_{a,c}))$; since $\frac{y_j-y^k}{r_k}$
is a free boundary element for $\tilde u_k$, that contradicts the continuity of $F$.   
 If $$S=\{0 < x_2 < \eta, -R+\frac{-c}{a}< x_1 < R+\frac{-c}{a}\},$$ observe $v:=\partial_{x_1} \tilde u_k\ge 0$ on $\partial S$. By differentiating, $Lv=0$ on $S \cap \Omega_k$, where $L$ is a linear second order uniformly elliptic operator. Next, since $v$ vanishes on $\partial \Omega_k$, it follows that $v>0$ in $S \cap \Omega_k$ with the maximum principle. Note that for a disk $B \subset B_{R/2}^+(x_{a,c})$ such that $\nabla \tilde u_k =0$ in $B$, $u_k=m$ for some $m \in \mathbb{R}$ and there is a strip $\tilde S$ generated  via  translating the disk in the $x_1$-direction. Choosing another disk 
$$\tilde B \subset \tilde S \cap \{-R +\frac{-c}{a}< x_1 < -\frac{R}{2}+\frac{-c}{a}\},$$ 
let $E_t=\tilde B+te_1$ for $t \in \mathbb{R}$;  $v=0$ in $\Omega_k^c$ and $v>0$ in $S \cap \Omega_k$, therefore it follows that $u < m$ on $\tilde B$. Next, assume $t^*>0$ denotes the first value for which $\partial E_t$ intersects $\{\tilde u_k=m\}$ and let 
$$y \in \partial E_{t^*} \cap \{\tilde u_k=m\}.$$ Via $F(D^2 \tilde u_k) \ge 0$ and $w=\tilde u_k-m<0$ in $E_{t^*}$ with $w(y)=0$, Hopf's principle  implies (see e.g. \cite{MR2994551}) $\partial_{\mathcal{N}} \tilde u_k(y) >0$. In particular, there is $\mu>0$ such that $B_\mu(y) \subset \Omega_k$. Therefore $v>0$ on $B_\mu(y)$ and since $v = 0$ on $\Omega_k^c$, there is $z>0$ such that $\tilde u_k(y+e_1z)>m$ for $y+e_1z \in \{\tilde u_k = m\}$, a contradiction.
\end{proof}

This implies $C^1$ regularity in two dimensions without sign assumptions and thus solves the problem in \cite[Chapter 8, Remark 8.8]{MR2962060}.

\begin{thm}  \label{1r}
Let $u \in P_1^+(0, M, \Omega)$,  $0 \in \overline{\Gamma(u)},$ and assume that $n=2$.
Then there exists $r_0>0$ such that $\Gamma$ is the graph of a $C^1$ function in $B_{r_0}^+$. 
\end{thm}

\bibliographystyle{alpha}

\bibliography{ngonref2}

\end{document}